\documentclass[english]{amsart}
\usepackage{a4wide}

\newtheorem{theorem}{Theorem}[section]
\newtheorem{lemma}[theorem]{Lemma}

\newtheorem{corollary}[theorem]{Corollary}

\theoremstyle{definition}
\newtheorem{definition}[theorem]{Definition}
\newtheorem{example}[theorem]{Example}

\numberwithin{equation}{section}

\theoremstyle{remark}

\usepackage{graphicx}

\usepackage[fixlanguage]{babelbib}
\selectbiblanguage{english}

\usepackage[shortlabels]{enumitem}

\usepackage{amsfonts,amssymb,amsbsy,amsmath,mathrsfs,amsthm}

\usepackage{commath}

\usepackage{hyperref}
\usepackage{url}

\hypersetup{pdfpagemode=UseNone, pdfstartview=FitH,
  colorlinks=true,urlcolor=red,linkcolor=blue,citecolor=black,
  pdftitle={Default Title, Modify},pdfauthor={Your Name},
  pdfkeywords={Modify keywords}}

\RequirePackage{babel}

\newcommand{\dmu}{\, d \mu}

\newcommand{\dla}{\, d \lambda}

\def\Xint#1{\mathchoice
   {\XXint\displaystyle\textstyle{#1}}%
   {\XXint\textstyle\scriptstyle{#1}}%
   {\XXint\scriptstyle\scriptscriptstyle{#1}}%
   {\XXint\scriptscriptstyle\scriptscriptstyle{#1}}%
   \!\int}
\def\XXint#1#2#3{{\setbox0=\hbox{$#1{#2#3}{\int}$}
     \vcenter{\hbox{$#2#3$}}\kern-.5\wd0}}

\def\dashint{\Xint-}

\usepackage{cite}
\makeatletter
\newcommand{\citecomment}[2][]{\citen{#2}#1\citevar}
\newcommand{\citeone}[1]{\citecomment{#1}}
\newcommand{\citetwo}[2][]{\citecomment[,~#1]{#2}}
\newcommand{\citevar}{\@ifnextchar\bgroup{;~\citeone}{\@ifnextchar[{;~\citetwo}{]}}}
\newcommand{\citefirst}{\@ifnextchar\bgroup{\citeone}{\@ifnextchar[{\citetwo}{]}}}

\makeatother

\begin{document}

\title{Weak Gurov--Reshetnyak class in metric measure spaces}

\author{Kim Myyryl\"ainen}
\address{Department of Mathematics, Aalto University, P.O. Box 11100, FI-00076 Aalto, Finland}
\email{kim.myyrylainen@aalto.fi}
\thanks{The research was supported by the Magnus Ehrnrooth Foundation.
The author would like to thank Juha Kinnunen for valuable discussions.
The author would also like to thank the anonymous referee for carefully reading the paper and for constructive comments.
}

\subjclass[2020]{42B35, 43A85}

\keywords{Gurov--Reshetnyak class, reverse H\"older inequality, doubling measure, metric space}

\begin{abstract}
We introduce a weak Gurov--Reshetnyak class and discuss its connections to 
a weak Muckenhoupt $A_\infty$ condition and
a weak reverse H\"older inequality in the setting of metric measure spaces with a doubling measure.
A John--Nirenberg type lemma 
is shown
for the weak Gurov--Reshetnyak class 
which gives a specific decay estimate for the oscillation of a function.
It implies that a function in the weak Gurov--Reshetnyak class satisfies the weak reverse H\"older inequality.
This comes with an upper bound for the reverse H\"older exponent depending on the Gurov--Reshetnyak parameter
which allows
the study of the asymptotic behavior of 
the exponent.

\end{abstract}

\maketitle

\section{Introduction}

The Gurov--Reshetnyak class
was first 
introduced and 
studied by Gurov and Reshetnyak~\cite{gurov1975,gurovreshetnyak1976}
in the context of quasiconformal mappings.
In addition, it has applications in the theory of PDEs and reverse H\"older inequalities, see for example~\cite{bojarski1985,iwaniec1982,reshetnyak1994}.
In 2002, Korenovskyy, Lerner and Stokolos~\cite{KorenovskyyLernerStokolos2002}
showed that the
Gurov--Reshetnyak condition $GR_\varepsilon$, $0<\varepsilon<2$,
\[
\int_{B} \lvert w-w_{B} \lvert \dmu \leq \varepsilon w(B)
\]
characterizes the Muckenhoupt $A_\infty$ condition
\[
w(B \cap \{ \alpha w \geq w_{B} \}) \leq \beta w(B)
\]
in the Euclidean setting,
see also~\cite{Korenovskii2007}.
We observe that their argument works not only in the Euclidean setting but also in metric measure spaces.
Here $w$ is a nonnegative locally integrable function with respect to the measure $\mu$, $B$ is a ball,
$w_B$ is the integral average of $w$ in a ball $B$ 
and $\alpha,\beta \in (0,1)$.
With a slight abuse of notation $w$ also denotes the measure induced by the weight.
The $A_\infty$ condition
is equivalent to the fact that there exists $1<p<\infty$ such that the reverse H\"older inequality
\[
\biggl( \dashint_{B} w^p \dmu \biggr)^\frac{1}{p} \leq C \dashint_{B} w \dmu
\]
holds.
There exist many 
other equivalent conditions, see for example~\cite{kortekansanen2011,duo2013,duomartinombrosi2016}.
However,
extra care is needed 
in metric measure spaces with a doubling measure $\mu$ 
since 
the equivalence of some conditions might fail
without 
an additional geometrical assumption
on the space, for example
the annular decay property 
or the continuity of the measure of a ball with respect to its radius, see~\cite{strombergtorchinsky1989,kinnunenshukla2014a,kinnunenshukla2014b}.

This paper extends 
the discussed results in the following direction.
Let $\mu$ be a doubling measure in a metric measure space,
$\sigma>1$ and $0<\varepsilon<1$.
Here we do not need to impose any additional assumptions on the space.
We show that the following notion of the weak Gurov--Reshetnyak class
\begin{equation}
\label{intro:GR}
\int_{B} (w - w_{\sigma B})_+ \dmu \leq \varepsilon w(\sigma B)
\end{equation}
characterizes the weak $A_\infty$ condition
\[
w(B \cap \{ \alpha w \geq w_{\sigma B} \}) \leq \beta w(\sigma B)
\]
in Theorem~\ref{weakGRequiv}.
This in turn 
coincides with the weak reverse H\"older inequality
\begin{equation*}
\biggl( \dashint_{B} w^p \dmu \biggr)^\frac{1}{p} \leq C \dashint_{\sigma B} w \dmu
\end{equation*}
with $p>1$ and $C>0$,
whenever the parameters in the weak $A_\infty$ condition are restricted depending on the measure, see~\cite{KinnunenKurkiMudarra2022}.
The weak $A_\infty$ and the weak reverse H\"older inequality have recently been 
discussed 
in~\cite{andersonhytonentapiola2017,KinnunenKurkiMudarra2022,Spadaro2012,IndratnoMaldonadoSilwal2015,hytonenperezrela2012}.
Surprisingly, it is enough to have the positive part of the oscillation in~\eqref{intro:GR}.
Note that the definition of \eqref{intro:GR} with the absolute value implies the definition with the positive part. 
Thus, \eqref{intro:GR} covers a wider class of functions and is a natural definition in metric measure spaces.
In Theorem~\ref{GRneg}, we observe that the corresponding definition with the negative part
of the oscillation
coincides with a different weak $A_\infty$ condition 
where we have the measure $\mu$ and the sublevel set instead of the weighted measure $w$ and the superlevel set.
The proofs of Theorems~\ref{weakGRequiv} and~\ref{GRneg}
are inspired by~\cite{KorenovskyyLernerStokolos2002}.
Other weak 
formulations of the Gurov--Reshetnyak class have been considered in~\cite{iwaniec1982,KorenovskyyLernerStokolos2007,bojarski1989,BerkovitsKinnunenMartell}.

A particularly useful result is the John--Nirenberg lemma for $GR_\varepsilon$ 
which gives a specific decay estimate for the mean oscillation of a function in $GR_\varepsilon$.
It tells that
a $GR_\varepsilon$ function
is locally integrable to a higher power $p>1$ depending on $\varepsilon$.
This has been studied in the Euclidean setting in \cite{bojarski1985,bojarski1989,franciosi1989,franciosi1991,gurov1975,gurovreshetnyak1976,iwaniec1982,korenovskii1990,wik1987} and in metric measure spaces in \cite{AaltoBerkovits,BerkovitsKinnunenMartell}.
In Theorem~\ref{reshetnyak},
we show a John--Nirenberg lemma for the weak Gurov--Reshetnyak class.
The argument is based on a Calder\'{o}n--Zygmund decomposition in metric measure spaces.
The John--Nirenberg lemma implies that the weak Gurov--Reshetnyak class is self-improving, in particular, 
a function
is locally integrable to a higher power $p>1$ with reverse H\"older type bounds, see Corollary~\ref{oscRHI}.
Another consequence 
is the weak reverse H\"older inequality in Corollary~\ref{reverseHolder1} with the exponent $p>1$ depending on $\varepsilon$.
Our 
approach 
allows the study of the asymptotic behavior of
the previous results,
in particular,
the growth of the upper bound of the exponent $p$ to infinity when $\varepsilon$ tends to zero.
This is analogous to the classical setting.
Moreover, our approach gives an alternative proof for the classical Gurov--Reshetnyak self-improvement in metric measure spaces in addition to~\cite{AaltoBerkovits, BerkovitsKinnunenMartell}
and extends them by assuming weaker condition~\eqref{intro:GR}.
To our knowledge, 
definition~\eqref{intro:GR} and concerning results are new even in the Euclidean setting.

\section{Weak Gurov--Reshetnyak class}

Let $(X,d,\mu)$ be a metric measure space with a metric $d$ and a doubling measure~$\mu$. 
A Borel regular measure is said to be doubling if
\[
0 < \mu(2B) \leq C_\mu \mu(B) < \infty
\]
for every ball $B = B(x,r) = \{ y\in X: d(x,y) < r \}$, where $C_\mu>1$ is the doubling constant.
We use the notation $\lambda B = B(x,\lambda r)$, $\lambda>0$, for the $\lambda$-dilate of $B$.
From the doubling property of the measure, it can be deduced that if $y \in B(x,R) \subset X$ and $0<r\leq R <\infty$, then
\begin{equation}
\label{eq:measureratio}
\frac{\mu(B(x,R))}{\mu(B(y,r))} \leq C_\mu^2 \biggl( \frac{R}{r} \biggr)^D,
\end{equation}
where $D = \log_2 C_\mu$ is the doubling dimension of the space $(X,d,\mu)$. The proof can be found in~\cite[p.~6]{bjorn}.
Unless otherwise stated, constants are positive and the dependencies on parameters are indicated in the brackets.

The integral average of $f \in L^1(A)$ over a measurable set $A\subset X$, with $0<\mu(A)<\infty$, is denoted by
\[
f_A = \dashint_A f \dmu = \frac{1}{\mu(A)} \int_A f \dmu .
\]
The positive and the negative parts of a function $f$ are denoted by
\[
f_+ = \max\{ f, 0 \} 
\quad\text{and}\quad
f_- = - \min\{ f , 0 \} .
\]

This section discusses the weak Gurov--Reshetnyak class and its relation to the weak reverse H\"older inequality. 

\begin{definition}
\label{weakGR}

Let $\Omega\subset X$ be an open set, $0<\varepsilon<1$ and $\sigma \geq 1$.
A nonnegative locally integrable function 
$w\in L^1_{\text{\normalfont loc}}(\Omega)$
belongs to the weak Gurov--Reshetnyak class denoted by
$WGR_{\varepsilon,\sigma}(\Omega)$
if
\[
\int_{B} (w - w_{\sigma B})_+ \dmu \leq \varepsilon w(\sigma B)
\]
for every ball 
$ \sigma B \subset \Omega$.
If $\sigma=2$, we denote $WGR_{\varepsilon}(\Omega) = WGR_{\varepsilon,2}(\Omega)$.

\end{definition}

The weak Gurov--Reshetnyak class is appropriately called weak since for $w \in GR_\varepsilon$, $ 0<\varepsilon<2$, we have
\[
\int_B (w-w_{\sigma B})_+ \dmu \leq \int_{\sigma B} (w-w_{\sigma B})_+ \dmu = \frac{1}{2} \int_{\sigma B} \lvert w-w_{\sigma B} \lvert \dmu \leq \frac{\varepsilon}{2} w(\sigma B),
\]
that is, $w \in WGR_{\varepsilon/2,\sigma} $ for every $\sigma>1$.
Thus, it holds that $GR_\varepsilon \subset WGR_{\varepsilon/2,\sigma}$.
The inclusion is proper since by considering the example $w(x)=e^x$ on $\mathbb{R}$ with the Lebesgue measure
we see that
$w\in WGR_{\varepsilon/2,\sigma} \setminus GR_\varepsilon$.
Moreover, observe that $WGR_{\varepsilon/2,1} = GR_{\varepsilon}$ since
\[
\int_B (w-w_{B})_+ \dmu = \frac{1}{2} \int_{ B} \lvert w-w_{ B} \lvert \dmu \leq \frac{\varepsilon}{2} w(B) .
\]

The following theorem shows that the weak Gurov--Reshetnyak class with $0<\varepsilon<1$ coincides with
superlevel set measure condition
\eqref{weakA} 
for $\alpha,\beta \in (0,1)$.

\begin{theorem}
\label{weakGRequiv}
Let $B\subset X$ be a ball, $\sigma\geq1$ and
$w$ be a nonnegative locally integrable function.
\begin{enumerate}[(i)]
\item 
Assume that there exists $0 < \varepsilon < 1$ such that
\[
\int_{B} (w - w_{\sigma B})_+ \dmu \leq \varepsilon w(\sigma B) .
\]
Then for $\varepsilon < \lambda < 1 $ we have
\[
w(B \cap \{ (1 - \tfrac{\varepsilon}{\lambda} ) w \geq w_{\sigma B} \}) \leq \lambda w(\sigma B) .
\]

\item
Assume that there exist $\alpha,\beta \in (0,1)$ such that
\begin{equation}
\label{weakA}
w(B \cap \{ \alpha w \geq w_{\sigma B} \}) \leq \beta w(\sigma B) .
\end{equation}
Then we have
\[
\int_{B} (w - w_{\sigma B})_+ \dmu \leq (1-\alpha (1-\beta)) w(\sigma B) .
\]

\end{enumerate}
\end{theorem}

\begin{proof}
We first show that (i) holds.
Let $E = B \cap \{ (1- \tfrac{\varepsilon}{\lambda} ) w \geq w_{\sigma B} \}$.
We obtain
\begin{align*}
\frac{\varepsilon}{\lambda} w(E) &= \frac{\varepsilon}{\lambda} \int_{E} w \dmu \leq \int_{E} (w-w_{\sigma B}) \dmu \leq \int_{B \cap \{ w \geq w_{\sigma B} \}} (w-w_{\sigma B}) \dmu \\
& = \int_{B} (w-w_{\sigma B})_+ \dmu \leq\varepsilon w(\sigma B) ,
\end{align*}
which implies that $w(E) \leq \lambda w(\sigma B) $.

For the other direction, we set $E = B \cap \{ \alpha w \geq w_{\sigma B} \}$
and $E^c = B \setminus E = B \cap \{ \alpha w < w_{\sigma B} \}$.
Then $w(E) \leq \beta w(\sigma B)$ and it holds that
\begin{align*}
\int_{B} (w - w_{\sigma B})_+ \dmu &= \int_{B \cap \{ w > w_{\sigma B} \}} (w-w_{\sigma B}) \dmu \\
&=\int_{B \cap \{ w_{\sigma B} < w < \frac{1}{\alpha} w_{\sigma B} \}} (w-w_{\sigma B}) \dmu + \int_{E} (w-w_{\sigma B}) \dmu \\
&\leq (1-\alpha) w(E^c)  + w(E) \\
&= (1-\alpha) w(B)  + \alpha w(E) \\
&\leq (1-\alpha) w(\sigma B)  + \alpha \beta w(\sigma B) \\
&= (1-\alpha (1-\beta)) w(\sigma B).
\end{align*}
This completes the proof.
\end{proof}

By the next lemma,
we know that
superlevel set measure condition
\eqref{weakA} characterizes the weak reverse H\"older inequality when $\beta<C_\mu^{-\lfloor \log_2(5\sigma^2) \rfloor-1}$,
see \cite[Theorem~4.4]{KinnunenKurkiMudarra2022}.
If we restrict $0<\varepsilon<C_\mu^{-\lfloor \log_2(5\sigma^2) \rfloor-1}$, then the argument in (i) of Theorem~\ref{weakGRequiv} works and we obtain condition \eqref{weakA} with $\beta<C_\mu^{-\lfloor \log_2(5\sigma^2) \rfloor-1}$
and thus the weak reverse H\"older inequality.

\begin{lemma}
\label{wRHI-quali}
Let $\Omega\subset X$ be an open set,
$\sigma>1$ and $w$ be a nonnegative locally integrable function.
The following statements are equivalent.
\begin{enumerate}[(i)]
\item 
There exist $p>1$ and a constant $C>0$ such that
\[
\biggl( \dashint_{B} w^p \dmu \biggr)^\frac{1}{p} \leq C \dashint_{\sigma B} w \dmu
\]
for every ball $\sigma B\subset \Omega$.
\item
There exist $\alpha,\beta>0$ with 
$\beta < C_\mu^{-\lfloor \log_2(5\sigma^2) \rfloor-1}$
such that
\begin{equation*}
w(B \cap \{ \alpha w \geq w_{\sigma B} \}) \leq \beta w(\sigma B)
\end{equation*}
for every ball $\sigma B\subset \Omega$.
\end{enumerate}

\end{lemma}

The following example shows that the parameter $\varepsilon$ in the definition of the weak Gurov--Reshetnyak class indeed needs to be restricted for it to 
imply the weak reverse H\"older inequality.
The example is due to Sawyer in $\mathbb{R}^2$~\cite{sawyer1982}.

\begin{example}
\label{examplesawyer}
Consider $\mathbb{R}^n$ with the Lebesgue measure $\mu$
and
the weak Gurov--Reshetnyak class 
with respect to cubes and $\sigma = 2$.
Let $S = \{x\in \mathbb{R}^n : 0\leq x_n \leq 1 \}$
and
$w=\chi_S$.
By \cite[Example~4.2]{KinnunenKurkiMudarra2022}, we know that
$w$ does not satisfy a weak reverse H\"older inequality.
We have
\[
\frac{w(Q)}{w(2Q)} = \frac{\mu(Q\cap S)}{\mu(2Q \cap S)} \leq \frac{1}{2^{n-1}}
\]
for every cube $Q\subset\mathbb{R}^n$.
It follows that
\begin{align*}
\int_Q (w-w_{2Q})_+ \dmu
&\leq w(Q) 
= \frac{w(Q)}{w(2Q)}  w(2Q)
\leq \frac{1}{2^{n-1}} w(2Q) .
\end{align*}
Thus, $w\in WGR_\varepsilon$ with $\varepsilon=2^{1-n}$.
This shows that for the weak Gurov--Reshetnyak condition to imply the weak reverse H\"older inequality, the parameter $\varepsilon$ needs to be restricted,
at least smaller than $2^{1-n}$.
Compare to Theorem~\ref{reshetnyak} below where we have the restriction $\varepsilon < 1 /(C_\mu (5\sigma)^D e )$.
\end{example}

The following theorem shows that a weak Gurov--Reshetnyak condition with the negative part of the oscillation is equivalent 
with 
sublevel set measure condition \eqref{measurecond_not_weakA}.
We know that this sublevel set condition
does not characterize the weak reverse H\"older inequality, see (*b) in~\cite[p.~2283]{KinnunenKurkiMudarra2022}.
Thus, we see that the 
correct definition of the weak Gurov--Reshetnyak class related to the weak reverse H\"older inequality only has
the positive part of the oscillation.

\begin{theorem}
\label{GRneg}
Let $B\subset X$ be a ball and $w$ be a nonnegative locally integrable function.
\begin{enumerate}[(i)]
\item 
Assume that there exists $0 < \varepsilon < 1$ such that
\[
\dashint_{B} (w - w_{\sigma B})_- \dmu \leq \varepsilon w_{\sigma B} .
\]
Then for $\varepsilon < \lambda < 1 $ we have
\[
\mu( B \cap \{ w \leq (1- \tfrac{\varepsilon}{\lambda} ) w_{\sigma B} \}) \leq \lambda \mu(B) .
\]

\item
Assume that there exist $\alpha,\beta \in (0,1)$ such that
\begin{equation}
\label{measurecond_not_weakA}
\mu( B \cap \{ w \leq \beta w_{\sigma B} \} ) \leq \alpha \mu(B) .
\end{equation}
Then we have
\[
\dashint_{B} (w - w_{\sigma B})_- \dmu \leq (1-(1-\alpha) \beta) w_{\sigma B} .
\]

\end{enumerate}
\end{theorem}

\begin{proof}
We first show that (i) holds.
Let $E = B \cap \{ w \leq (1- \tfrac{\varepsilon}{\lambda} ) w_{\sigma B} \}$.
We obtain
\begin{align*}
\frac{\varepsilon}{\lambda} w_{\sigma B} &\leq \inf_{E} (w_{\sigma B} - w) \leq \dashint_{E} (w_{\sigma B} - w) \dmu \leq \frac{1}{\mu(E)} \int_{B \cap \{ w \leq w_{\sigma B} \}} (w_{\sigma B} - w) \dmu \\
& = \frac{1}{\mu(E)} \int_{B} (w_{\sigma B} - w)_+ \dmu \leq \frac{\mu(B)}{\mu(E)} \varepsilon w_{\sigma B} ,
\end{align*}
which implies that $\mu(E) \leq \lambda \mu(B) $.

For the other direction, we set $E = B \cap \{ w \leq \beta w_{\sigma B} \}$
and $E^c = B \setminus E = B \cap \{ w > \beta w_{\sigma B} \}$.
Then $\mu(E) \leq \alpha \mu(B)$ and it holds that
\begin{align*}
\dashint_{B} (w - w_{\sigma B})_- \dmu &= \frac{1}{\mu(B)} \int_{ B \cap \{ w < w_{\sigma B} \}} (w_{\sigma B} - w) \dmu \\
&= \frac{1}{\mu(B)} \int_{B \cap \{ \beta w_{\sigma B} < w < w_{\sigma B} \}} (w_{\sigma B} - w) \dmu + \frac{1}{\mu(B)} \int_{E} (w_{\sigma B} - w) \dmu \\
&\leq \frac{\mu(E^c)}{\mu(B)} (1-\beta) w_{\sigma B} + \frac{\mu(E)}{\mu(B)} w_{\sigma B} \\
&= \frac{w_{\sigma B}}{\mu(B)} ( (1-\beta) \mu(B) + \beta \mu(E) ) \\
&\leq \frac{w_{\sigma B}}{\mu(B)}  ( (1-\beta) \mu(B) + \beta \alpha \mu(B) ) \\
&= (1-(1-\alpha) \beta) w_{\sigma B} .
\end{align*}
This completes the proof.
\end{proof}

\section{John--Nirenberg lemma for weak Gurov--Reshetnyak class}

Throughout the argument, let 
$\sigma\geq1$,
$\eta > 0$ and $B_0 = B(x_{B_0},r_{B_0}) \subset X$ be fixed. We denote
$\widehat{B}_0 = (1+\eta) B_0$,
\[
\mathcal{B} = \{B(x_B,r_B) : x_B \in B_0, r_B \leq \eta r_{B_0} \}
\]
and
\begin{equation}
\label{alpha}
\alpha = C_\mu^2 (5\sigma)^D \biggl( 1 + \frac{1}{\eta} \biggr)^D ,
\end{equation}
where recall that $C_\mu$ is the doubling constant and $D = \log_2 C_\mu$ is the doubling dimension.
We define a maximal function
\[
M_{\mathcal{B}}f(x) = \sup_{\substack{B \ni x \\ B \in \mathcal{B}}} \dashint_{B} \lvert f \rvert \dmu
\]
with the understanding that $M_\mathcal{B} f(x) = 0$ if there is no ball $B \in \mathcal{B}$ such that $x \in B$. 
In particular, $M_\mathcal{B} f(x) = 0$ for every $x \in X \setminus \widehat{B}_0$. By the Lebesgue differentiation theorem~\cite[p.~4]{heinonen},
we have $|f(x)| \leq M_{\mathcal{B}}f(x)$ for $\mu$-almost every $x \in B_0$.
Moreover, denote 
\[
E_\lambda = \{ x \in \widehat{B}_0 : M_{\mathcal{B}}f(x) > \lambda \} .
\]

The following lemma is a Calder\'{o}n--Zygmund decomposition in metric measure spaces with a doubling measure. See~\cite{BerkovitsKinnunenMartell, myyrylainen} for the proof.

\begin{lemma}
\label{calderon}
Let $f\geq 0$ be an integrable function defined on $\widehat{B}_0$.
Assume that $E_\lambda \neq \emptyset $ and
\begin{equation*}
\alpha f_{\widehat{B}_0} \leq \lambda,
\end{equation*}
where $\alpha$ is given in~\eqref{alpha}.
Then there exist countably many pairwise disjoint balls $B_i \in \mathcal{B} $ such that
\begin{enumerate}[(i),topsep=5pt,itemsep=5pt]
\item $ \bigcup_i B_i \subset E_\lambda \subset \bigcup_i 5B_i $,

\item $r_{B_i} \leq \frac{\eta}{5\sigma} r_{B_0}$,

\item $ f_{B_i} > \lambda$,

\item $ f_{\tau B_i} \leq \lambda$ whenever $ \tau \geq 2$ and $\tau B_i \in \mathcal{B}$.

\end{enumerate}
The collection of balls $\{B_i\}_i$ is called the Calder\'{o}n--Zygmund balls $B_{i,\lambda}$ at level $\lambda$. Furthermore, if $\alpha f_{\widehat{B}_0} \leq \lambda' \leq \lambda$, then it is possible to choose Calder\'{o}n--Zygmund balls $ B_{j,\lambda'} $ at level $\lambda'$ in a manner that for each $B_{i,\lambda}$ we can find $B_{j, \lambda'}$ such that
$
B_{i,\lambda} \subset 5 B_{j,\lambda'} .
$

\end{lemma}

The following theorem is the John--Nirenberg lemma for the weak Gurov--Reshetnyak class.

\begin{theorem}
\label{reshetnyak}
Let $0<\varepsilon<1$.
Assume that $w \in WGR_{\varepsilon,\sigma}(\sigma \widehat{B}_0)$.
Then there exists a constant $C = C(C_\mu, \sigma, \eta)$ such that
\begin{align*}
\mu(\{ x \in B_0: ( w-w_{\sigma \widehat{B}_0} )_+ > \lambda w_{\sigma \widehat{B}_0} \}) \leq \bigg( \frac{1}{1+\lambda} \bigg)^{1/(A\varepsilon)} \frac{C}{ \varepsilon w_{\sigma \widehat{B}_0}} \int_{\widehat{B}_0} ( w-w_{\sigma \widehat{B}_0} )_+ \dmu 
\end{align*}
for every $\lambda\geq \lambda_0 = \alpha C_\mu \sigma^D \varepsilon$, where $\alpha$ is given in~\eqref{alpha}.

\end{theorem}

\begin{proof}

We have
\begin{align*}
\dashint_{\widehat{B}_0} ( w-w_{\sigma \widehat{B}_0} )_+ \dmu \leq \varepsilon \frac{w(\sigma \widehat{B}_0)}{\mu(\widehat{B}_0)} \leq C_\mu \sigma^D \varepsilon w_{\sigma \widehat{B}_0} \leq  \frac{\lambda}{\alpha} w_{\sigma \widehat{B}_0} 
\end{align*}
for $\lambda \geq \lambda_0 = \alpha C_\mu \sigma^D \varepsilon$,
since
\[
\frac{\mu(\sigma \widehat{B}_0)}{\mu(\widehat{B}_0)} \leq C_\mu \sigma^D .
\]
Thus, we may apply Lemma~\ref{calderon}. Take $\lambda > \delta \geq \lambda_0 $ and form collections of Calder\'{o}n--Zygmund balls $\{B_i\}_i$ and $\{B_j\}_j$ for $( w-w_{\sigma \widehat{B}_0} )_+$ at levels $\lambda w_{\sigma \widehat{B}_0} $ and $\delta w_{\sigma \widehat{B}_0}$ such that each $B_i$ is contained in some $5B_j$.
Denote
\[
\mathcal{I}_j = \biggl\{ i \in \mathbb{N}: B_{i} \subset 5B_{j}, i \notin \bigcup_{k=1}^{j-1} \mathcal{I}_k \biggr\}
\]
for every $j \in \mathbb{N}$.
We have $5 \sigma B_j \in \mathcal{B}$ due to $5\sigma r_{B_j} \leq 5\sigma \frac{\eta}{5\sigma} r_{B_0} = \eta r_{B_0}$
by Lemma~\ref{calderon}~(ii).
Lemma~\ref{calderon}~(iv) then implies
\[
( w_{5 \sigma B_j}-w_{\sigma \widehat{B}_0} )_+ \leq \dashint_{5 \sigma B_j} (w-w_{\sigma \widehat{B}_0})_+ \dmu \leq \delta w_{\sigma \widehat{B}_0} .
\]
Moreover,
Lemma~\ref{calderon}~(iii) gives
\begin{align*}
\lambda w_{\sigma \widehat{B}_0} \sum_{i=1}^\infty \mu(B_i) \leq \sum_{i=1}^\infty \int_{B_i} ( w-w_{\sigma \widehat{B}_0} )_+ \dmu = \sum_{j=1}^\infty \sum_{i \in \mathcal{I}_j} \int_{B_i} ( w-w_{\sigma \widehat{B}_0} )_+ \dmu .
\end{align*}
For each fixed $j\in\mathbb{N}$, we estimate
\begin{align*}
\sum_{i \in \mathcal{I}_j} \int_{B_i} ( w-w_{\sigma \widehat{B}_0} )_+ \dmu &\leq \sum_{i \in \mathcal{I}_j} \int_{B_i} ( w-w_{5 \sigma B_j} )_+ \dmu + \sum_{i \in \mathcal{I}_j} \int_{B_i} ( w_{5 \sigma B_j}-w_{\sigma \widehat{B}_0} )_+ \dmu \\
&\leq \int_{5B_j} ( w-w_{5 \sigma B_j} )_+ \dmu + \delta w_{\sigma \widehat{B}_0} \sum_{i \in \mathcal{I}_j} \mu( B_i ) \\
&\leq \varepsilon w(5\sigma B_j) + \delta w_{\sigma \widehat{B}_0} \sum_{i \in \mathcal{I}_j} \mu( B_i ) \\
&\leq \varepsilon (1+\delta) w_{\sigma \widehat{B}_0} \mu(5\sigma B_j) + \delta w_{\sigma \widehat{B}_0} \sum_{i \in \mathcal{I}_j} \mu( B_i ) \\
&\leq C_\mu (5\sigma)^D \varepsilon (1+\delta) w_{\sigma \widehat{B}_0} \mu(B_j) + \delta w_{\sigma \widehat{B}_0} \sum_{i \in \mathcal{I}_j} \mu( B_i ) ,
\end{align*}
where in the second last inequality we used
\[
w_{5 \sigma B_j} \leq \dashint_{5 \sigma B_j} (w-w_{\sigma \widehat{B}_0})_+ \dmu + w_{\sigma \widehat{B}_0} \leq (1+\delta) w_{\sigma \widehat{B}_0} .
\]
By summing over $j\in\mathbb{N}$, we obtain
\[
\lambda w_{\sigma \widehat{B}_0} \sum_{i=1}^\infty \mu(B_i) \leq C_\mu (5\sigma)^D \varepsilon (1+\delta) w_{\sigma \widehat{B}_0} \sum_{j=1}^\infty \mu(B_j) + \delta w_{\sigma \widehat{B}_0} \sum_{i=1}^\infty \mu(B_i) .
\]
Thus, we have
\begin{equation}
\label{goodla1}
\sum_{i=1}^\infty \mu(B_i) \leq \frac{C_\mu (5\sigma)^D \varepsilon (1+\delta)}{\lambda - \delta} \sum_{j=1}^\infty \mu(B_j) 
\end{equation}
for $\lambda > \delta \geq \lambda_0 $.

Denote $S_\lambda = \bigcup_{i=1}^\infty B_i $ and $S_\delta = \bigcup_{j=1}^\infty B_j$, 
where the collections of balls $\{B_i\}_i$ and $\{B_j\}_j$ are as in the beginning of the proof.
Let
\[
\varphi(\delta) = (A \varepsilon + 1) \delta + A \varepsilon ,
\]
where $A = C_\mu (5\sigma)^D e$.
We recursively define $\lambda_{m+1} = \varphi(\lambda_m)$ for $m\in\mathbb{N}$ and recall that $\lambda_0 = \alpha C_\mu \sigma^D \varepsilon$.
Fix $m\in\mathbb{N}$.
We apply Lemma~\ref{calderon} first to construct $S_{\lambda_{m}}$ and then to recursively construct
$S_{\lambda_{k}}$ based on $S_{\lambda_{k+1}}$ for every $k\in\{0,1,\dots,m-1\}$.
Then by~\eqref{goodla1} we have
\begin{equation}
\label{goodla2}
\mu(S_{\lambda_{k+1}}) \leq \frac{C_\mu (5\sigma)^D \varepsilon (1+\lambda_k)}{\lambda_{k+1} - \lambda_k} \mu(S_{\lambda_k})  = \frac{1}{e} \mu(S_{\lambda_k}) 
\end{equation}
for every $k\in\{0,1,\dots,m-1\}$.
We make some observations.
Since $(c_0 \varepsilon + 1)^{1/(A\varepsilon)}$, $c_0 >0$, is decreasing as a function of $\varepsilon$, we see that
\[
(A \varepsilon + 1)^{1/(A\varepsilon)} \leq \lim_{\varepsilon \to 0} (A \varepsilon + 1)^{1/(A\varepsilon)} = e^1
\]
and
\[
(\lambda_0 + 1)^{1/(A\varepsilon)} \leq \lim_{\varepsilon \to 0} (\alpha C_\mu \sigma^D \varepsilon + 1)^{1/(A\varepsilon)} = e^{\alpha/(5^D e)} .
\]
Thus, it holds that
\begin{equation}
\label{lambda_m_estimate}
\biggl( \frac{1}{1+\lambda_k} \biggr)^{1/(A\varepsilon)} = (A \varepsilon + 1)^{1/(A\varepsilon)} \biggl( \frac{1}{1+\lambda_{k+1}} \biggr)^{1/(A\varepsilon)} \leq e \biggl( \frac{1}{1+\lambda_{k+1}} \biggr)^{1/(A\varepsilon)} 
\end{equation}
for every $k\in\{0,1,\dots,m\}$.

We claim that
\begin{equation}
\label{recurs}
\mu(S_{\lambda_k}) \leq C_0  \biggl( \frac{1}{1+\lambda_{k+1}} \biggr)^{1/(A\varepsilon)} \mu(S_{\lambda_0})
\end{equation}
for every $k\in\{0,1,\dots,m\}$, where $C_0 = e^{1+\alpha/(5^D e)}$.
We prove the claim by induction.
First, note that the claim holds for $k=0$ since
\begin{align*}
\mu(S_{\lambda_0}) &= (A \varepsilon + 1)^{1/(A\varepsilon)} (\lambda_0 + 1)^{1/(A\varepsilon)} \biggl( \frac{1}{1+\lambda_1} \biggr)^{1/(A\varepsilon)}  \mu(S_{\lambda_0}) \\
&\leq e^{1+\alpha/(5^D e)} \biggl( \frac{1}{1+\lambda_1} \biggr)^{1/(A\varepsilon)} \mu(S_{\lambda_0}) = C_0 \biggl( \frac{1}{1+\lambda_1} \biggr)^{1/(A\varepsilon)} \mu(S_{\lambda_0}) .
\end{align*}
Assume then that~\eqref{recurs} holds for some $k\in\{0,1,\dots,m-1\}$.
We show that this implies~\eqref{recurs} for $k+1$. By~\eqref{goodla2} and~\eqref{lambda_m_estimate}, we get
\begin{align*}
\mu(S_{\lambda_{k+1}}) &\leq \frac{1}{e} \mu(S_{\lambda_{k}}) \leq \frac{C_0}{e} \bigg( \frac{1}{1+\lambda_{k+1}} \bigg)^{1/(A\varepsilon)} \mu(S_{\lambda_0}) \leq C_0 \bigg( \frac{1}{1+\lambda_{k+2}} \bigg)^{1/(A\varepsilon)} \mu(S_{\lambda_0}) .
\end{align*}
Hence, \eqref{recurs} holds for $k+1$.

Let $m\in\mathbb{N}$ such that
$\lambda_m \leq \lambda < \lambda_{m+1}$.
Denote by $B_{j,\lambda_m}$ the Calder\'{o}n--Zygmund balls at level $\lambda_m w_{\sigma \widehat{B}_0}$
and by $B_{i,\lambda_0}$ the Calder\'{o}n--Zygmund balls at level $\lambda_0 w_{\sigma \widehat{B}_0}$.
The previous estimate together with Lemma~\ref{calderon}~(i), (iii) implies
\begin{align*}
\mu(\{ x \in B_0: ( w-w_{\sigma \widehat{B}_0} )_+ > \lambda w_{\sigma \widehat{B}_0} \}) &\leq \mu(\{ x \in B_0: ( w-w_{\sigma \widehat{B}_0} )_+ > \lambda_m w_{\sigma \widehat{B}_0} \}) \\ &\leq \sum_{j=1}^\infty \mu(5B_{j,\lambda_m}) \leq C_\mu^3 \sum_{j=1}^\infty \mu(B_{j,\lambda_m}) = C_\mu^3 \mu(S_{\lambda_m}) \\
&\leq C_\mu^3 C_0 \bigg( \frac{1}{1+\lambda_{m+1}} \bigg)^{1/(A\varepsilon)} \mu(S_{\lambda_0}) \\
&\leq C_\mu^3 C_0 \bigg( \frac{1}{1+\lambda} \bigg)^{1/(A\varepsilon)} \sum_{i=1}^\infty \mu(B_{i,\lambda_0}) \\
&\leq C_\mu^3 C_0 \bigg( \frac{1}{1+\lambda} \bigg)^{1/(A\varepsilon)} \frac{1}{\lambda_0 w_{\sigma \widehat{B}_0}} \sum_{i=1}^\infty \int_{B_{i,\lambda_0}} ( w-w_{\sigma \widehat{B}_0} )_+ \dmu \\
&\leq C_\mu^3 C_0 \bigg( \frac{1}{1+\lambda} \bigg)^{1/(A\varepsilon)} \frac{1}{\alpha C_\mu \sigma^D \varepsilon w_{\sigma \widehat{B}_0}} \int_{\widehat{B}_0} ( w-w_{\sigma \widehat{B}_0} )_+ \dmu \\
&= \bigg( \frac{1}{1+\lambda} \bigg)^{1/(A\varepsilon)} \frac{C}{ \varepsilon w_{\sigma \widehat{B}_0}} \int_{\widehat{B}_0} ( w-w_{\sigma \widehat{B}_0} )_+ \dmu ,
\end{align*}
where $C = C_\mu^2 C_0 / (\alpha \sigma^D)$.
This completes the proof.
\end{proof}

As a consequence of the preceding theorem, we obtain a self-improving phenomenon for the weak Gurov--Reshetnyak class.
In particular, functions in the weak Gurov--Reshetnyak class are integrable up to a power $p$ depending on $\varepsilon$.

\begin{corollary}
\label{oscRHI}
Let $0<\varepsilon<\frac{1}{2A}$,
where $A = C_\mu (5\sigma)^D e$.
Assume that $w \in WGR_{\varepsilon,\sigma}(\sigma\widehat{B}_0)$. 
Then for every $1<p\leq\frac{1}{2A\varepsilon}$ there exists a constant $C=C(C_\mu, \sigma, \eta, p)$ such that
\[
\int_{B_0} ( w-w_{\sigma \widehat{B}_0} )^p_+ \dmu \leq C \varepsilon^{p-1} (w_{\sigma \widehat{B}_0})^{p-1} \int_{\widehat{B}_0} ( w-w_{\sigma \widehat{B}_0} )_+ \dmu .
\]
\end{corollary}

\begin{proof}
Denote
\[
\Omega_\lambda = \{ x \in B_0: ( w-w_{\sigma \widehat{B}_0} )_+ > \lambda w_{\sigma \widehat{B}_0} \} .
\]
Applying Cavalieri's principle, we obtain
\begin{align*}
\int_{B_0} ( w-w_{\sigma \widehat{B}_0} )^p_+ \dmu = p (w_{\sigma \widehat{B}_0})^p \int_0^\infty \lambda^{p-1} \mu(\Omega_\lambda) \dla .
\end{align*}
For $0< \lambda < \lambda_0 = \alpha C_\mu \sigma^D \varepsilon $ we have
$\lambda^{p-1} \leq \lambda_0^{p-1}$,
and thus it holds that
\begin{align*}
\int_0^{\lambda_0} \lambda^{p-1} \mu(\Omega_\lambda) \dla &\leq \lambda_0^{p-1} \int_0^\infty \mu(\Omega_\lambda) \dla = (\alpha C_\mu \sigma^D)^{p-1} \varepsilon^{p-1} \frac{1}{w_{\sigma \widehat{B}_0}} \int_{B_0} ( w-w_{\sigma \widehat{B}_0} )_+ \dmu .
\end{align*}
On the other hand, for $\lambda \geq \lambda_0$ we may apply Theorem~\ref{reshetnyak} to get
a constant $C_0$ such that
\begin{align*}
\int_{\lambda_0}^\infty \lambda^{p-1} \mu(\Omega_\lambda) \dla &\leq \int_{\lambda_0}^\infty  \lambda^{p-1} \bigg( \frac{1}{1+\lambda} \bigg)^{1/(A\varepsilon)} \frac{C_0}{ \varepsilon w_{\sigma \widehat{B}_0}} \int_{\widehat{B}_0} ( w-w_{\sigma \widehat{B}_0} )_+ \dmu \dla \\
&\leq \int_{0}^\infty \frac{\lambda^{p-1}}{(1+\lambda)^{1/(A\varepsilon)}}  \dla \frac{C_0}{ \varepsilon w_{\sigma \widehat{B}_0}} \int_{\widehat{B}_0} ( w-w_{\sigma \widehat{B}_0} )_+ \dmu \\
&= B\Bigl(p,\tfrac{1}{A\varepsilon}-p\Bigr) \frac{C_0}{ \varepsilon w_{\sigma \widehat{B}_0}} \int_{\widehat{B}_0} ( w-w_{\sigma \widehat{B}_0} )_+ \dmu ,
\end{align*}
where 
$ A = C_\mu (5\sigma)^D e$
and $B$ is the beta function
\[
B\Bigl(p,\tfrac{1}{A\varepsilon}-p\Bigr) = \frac{\Gamma(p) \Gamma\bigl(\tfrac{1}{A\varepsilon}-p\bigr)}{\Gamma\bigl(\tfrac{1}{A\varepsilon}\bigr)} .
\]
By \cite[6.1.47]{abramowitz}, it holds that
\[
B(p,y-p) = \frac{\Gamma(p)}{y^p} \Bigl(1 + \mathcal{O}\bigl(\tfrac{1}{y}\bigr) \Bigr)
\]
as $y\to\infty$.
Assuming that $p\leq \frac{1}{2A\varepsilon}$, we have $\tfrac{1}{A\varepsilon}-p \geq \frac{1}{2A\varepsilon}$.
It follows that there exists a constant $C_1=C_1(p)$ such that
\[
B\Bigl(p,\tfrac{1}{A\varepsilon}-p\Bigr) \leq C_1 A^p \varepsilon^p
\]
for $0<\varepsilon<\frac{1}{2A}$ and $1<p\leq \frac{1}{2A\varepsilon}$.
Therefore, we have
\begin{align*}
\int_{\lambda_0}^\infty \lambda^{p-1} \mu(\Omega_\lambda) \dla &\leq 
C_1 A^p \varepsilon^{p-1} \frac{C_0}{ w_{\sigma \widehat{B}_0}} \int_{\widehat{B}_0} ( w-w_{\sigma \widehat{B}_0} )_+ \dmu .
\end{align*}
By combining the obtained estimates, we conclude that
\begin{align*}
\int_{B_0} ( w-w_{\sigma \widehat{B}_0} )^p_+ \dmu \leq C \varepsilon^{p-1} (w_{\sigma \widehat{B}_0})^{p-1} \int_{\widehat{B}_0} ( w-w_{\sigma \widehat{B}_0} )_+ \dmu ,
\end{align*}
where $C =  p(\alpha C_\mu \sigma^D)^{p-1}  + p C_1 A^p C_0 $.
\end{proof}

We may deduce a reverse H\"older inequality from the previous corollary. This tells that functions in the weak Gurov--Reshetnyak class satisfy a weak reverse H\"older inequality. Moreover, it shows the asymptotic behavior of the upper bound of the exponent $p$ at infinity when $\varepsilon$ tends to zero.

\begin{corollary}
\label{reverseHolder1}
Let $0<\varepsilon<\frac{1}{2A}$,
where $A = C_\mu (5\sigma)^D e$.
Assume that $w \in WGR_{\varepsilon,\sigma}(\sigma\widehat{B}_0)$.
Then for every $1<p\leq\frac{1}{2A\varepsilon}$ there exists a constant $C=C(C_\mu, \sigma, \eta, p)$ such that
\[
\biggl( \dashint_{B_0} w^p \dmu \biggr)^\frac{1}{p} \leq (C \varepsilon + 1) \dashint_{\sigma \widehat{B}_0} w \dmu . 
\]
\end{corollary}

\begin{proof}
By Corollary~\ref{oscRHI}, there exists a constant $C_0$ such that
\begin{align*}
\int_{B_0} ( w-w_{\sigma \widehat{B}_0} )^p_+ \dmu &\leq C_0 \varepsilon^{p-1} (w_{\sigma \widehat{B}_0})^{p-1} \int_{\widehat{B}_0} ( w-w_{\sigma \widehat{B}_0} )_+ \dmu 
\leq C_0 \varepsilon^{p-1} (w_{\sigma \widehat{B}_0})^{p-1} \varepsilon w(\sigma \widehat{B}_0) \\
&= C_0 \varepsilon^{p}  (w_{\sigma \widehat{B}_0})^{p} \mu(\sigma \widehat{B}_0) 
\leq C_0 \varepsilon^{p}  (w_{\sigma \widehat{B}_0})^{p} C_\mu ((1+\eta)\sigma)^D \mu(B_0) \\
&= C^p \varepsilon^{p}  (w_{\sigma \widehat{B}_0})^{p} \mu(B_0) ,
\end{align*}
where $C^p = C_0 C_\mu ((1+\eta)\sigma)^D$.
Thus, we have
\begin{align*}
\biggl( \dashint_{B_0} ( w-w_{\sigma \widehat{B}_0} )^p_+ \dmu \biggr)^\frac{1}{p} &\leq C \varepsilon w_{\sigma \widehat{B}_0} ,
\end{align*}
from which we can conclude the claim.
\end{proof}

It is possible to obtain a smaller ball on the right-hand side in Corollary~\ref{reverseHolder1}.
We need the following lemma for this.

\begin{lemma}
\label{decompose}
Let $B_0 = B(x_0,r_0) \subset X$ be a ball, $\sigma>1$ and $\eta>0$. There exist $N = N(C_\mu,\sigma,\eta)$ 
balls $B_i=B(x_i,r_i)$ such that $x_i \in B_0$, $B_0 \subset \bigcup_{i=1}^N B_i$, $r_i = \frac{\sigma - 1}{\sigma (1+\eta)} r_0 $ and $\sigma \widehat{B}_i = \sigma(1+\eta) B_i \subset \sigma B_0$.
\end{lemma}

\begin{proof}
Consider a collection of balls $\{ B(x, \tfrac{\rho}{5} r_0 ) \}_{x\in B_0}$, where $\rho = \frac{\sigma - 1}{\sigma (1+\eta)}$.
By the 5-covering theorem~\cite[p.~2]{heinonen}, there exist pairwise disjoint balls $\tfrac{1}{5} B_i = B(x_i, \tfrac{\rho}{5} r_0)$ such that $x_i \in B_0$ and $B_0 \subset \bigcup_i B_i$.
Since $x_i \in B_0$, it holds that $B_0 \subset B(x_i,2r_0)$.
The doubling property implies
\[
\mu(B_0) \leq \mu(B(x_i,2r_0)) \leq C_\mu \biggl( \frac{10}{\rho} \biggr)^D \mu(B(x_i,\tfrac{\rho}{5} r_0)) = C_\mu \biggl( \frac{10}{\rho} \biggr)^D \mu(\tfrac{1}{5} B_i) .
\]
Since $\tfrac{1}{5} B_i$ are pairwise disjoint and $\tfrac{1}{5} B_i = B(x_i,\tfrac{\rho}{5} r_0) \subset (1+\tfrac{\rho}{5})B_0 $, we have
\[
\sum_{i=1}^N \mu(\tfrac{1}{5} B_i) \leq \mu((1+\tfrac{\rho}{5})B_0) \leq C_\mu \biggl( 1+\frac{\rho}{5} \biggr)^D \mu(B_0) .
\]
By combining the two previous estimates, we obtain
\begin{align*}
N 
= \sum_{i=1}^N 1 \leq C_\mu^2 \biggl( \frac{10}{\rho} \Bigl( 1+\frac{\rho}{5} \Bigr) \biggr)^D = C_\mu^2 \biggl( \frac{10}{\rho} + 2 \biggr)^D = C_\mu^2 \biggl( \frac{10 \sigma(1+\eta) }{\sigma -1} + 2 \biggr)^D .
\end{align*}
Finally, we note that $\sigma \widehat{B}_i \subset \sigma B_0$ since
$r(\sigma \widehat{B}_i) = \sigma (1+\eta) \rho r_0 = (\sigma - 1) r_0$.
\end{proof}

\begin{corollary}
Let $0<\varepsilon<\frac{1}{2A}$,
where $A = C_\mu (5\sigma)^D e$.
Assume that $w \in WGR_{\varepsilon,\sigma}(\sigma\widehat{B}_0)$. 
Then for every $1<p\leq\frac{1}{2A\varepsilon}$ there exists a constant $C=C(C_\mu, \sigma, \eta, p)$ such that
\[
\biggl( \dashint_{B_0} w^p \dmu \biggr)^\frac{1}{p} \leq C \dashint_{\sigma B_0} w \dmu .
\]
\end{corollary}

\begin{proof}
Apply Lemma~\ref{decompose} and
Corollary~\ref{reverseHolder1} with the constant $C_0 = C\varepsilon+1$ to get
\begin{align*}
\biggl( \dashint_{B_0} w^p \dmu \biggr)^\frac{1}{p} &\leq \Biggl( \frac{1}{\mu(B_0)} \sum_{i=1}^N \int_{B_i} w^p \dmu \Biggr)^\frac{1}{p} = \Biggl( \sum_{i=1}^N \frac{\mu(B_i)}{\mu(B_0)}  \dashint_{B_i} w^p \dmu \Biggr)^\frac{1}{p} \\
&\leq  C_0 \Biggl( \sum_{i=1}^N \frac{\mu(2B_0)}{\mu(B_0)} \biggl( \dashint_{\sigma \widehat{B}_i} w \dmu  \biggr)^p \Biggr)^\frac{1}{p} \\
&\leq C_0 C_\mu^\frac{1}{p} \Biggl( \sum_{i=1}^N \frac{1}{\mu(\sigma \widehat{B}_i)^p} \Biggr)^\frac{1}{p} \int_{\sigma B_0} w \dmu \\
&\leq C_0 C_\mu^{2+\frac{1}{p}} \biggl( \frac{\sigma}{\sigma-1} \biggr)^{D} \Biggl( \sum_{i=1}^N \frac{1}{\mu(\sigma B_0)^p} \Biggr)^\frac{1}{p} \int_{\sigma B_0} w \dmu \\
&= C_0 C_\mu^{2+\frac{1}{p}} \biggl( \frac{\sigma}{\sigma-1} \biggr)^{D} N^\frac{1}{p} \dashint_{\sigma B_0} w \dmu ,
\end{align*}
where we also used $B_i\subset 2B_0$ and \eqref{eq:measureratio} in the form of
\[
\mu(\sigma B_0) \leq C_\mu^2 \bigg( \frac{\sigma}{\sigma-1} \bigg)^D \mu(\sigma \widehat{B}_i) .
\]
\end{proof}


\begin{thebibliography}{10}

\bibitem{AaltoBerkovits}
D.~Aalto and L.~Berkovits, \emph{Asymptotical stability of {M}uckenhoupt
  weights through {G}urov-{R}eshetnyak classes}, Trans. Amer. Math. Soc.
  \textbf{364} (2012), no.~12, 6671--6687.

\bibitem{abramowitz}
M.~Abramowitz and I.~A. Stegun, \emph{Handbook of mathematical functions with formulas, graphs, and mathematical tables}, Dover, New York, 1972.

\bibitem{andersonhytonentapiola2017}
T.~C. Anderson, T.~Hyt\"{o}nen, and O.~Tapiola, \emph{Weak {$A_\infty$} weights
  and weak reverse {H}\"{o}lder property in a space of homogeneous type}, J.
  Geom. Anal. \textbf{27} (2017), no.~1, 95--119.

\bibitem{BerkovitsKinnunenMartell}
L.~Berkovits, J.~Kinnunen, and J.~M. Martell, \emph{Oscillation estimates,
  self-improving results and good-{$\lambda$} inequalities}, J. Funct. Anal.
  \textbf{270} (2016), no.~9, 3559--3590.

\bibitem{bjorn}
A.~Bj\"{o}rn and J.~Bj\"{o}rn, \emph{Nonlinear potential theory on metric
  spaces}, EMS Tracts in Mathematics, vol.~17, European Mathematical Society
  (EMS), Z\"{u}rich, 2011.

\bibitem{bojarski1985}
B.~Bojarski, \emph{Remarks on the stability of reverse {H}\"{o}lder
  inequalities and quasiconformal mappings}, Ann. Acad. Sci. Fenn. Ser. A I
  Math. \textbf{10} (1985), 89--94.

\bibitem{bojarski1989}
\bysame, \emph{On {G}urov-{R}eshetnyak classes}, Bounded Mean Oscillation in
  complex analysis, University of Joensuu, Publications in sciences, no. 14,
  pp. 21--42, Joensuu, 1989.

\bibitem{duo2013}
J.~Duoandikoetxea, \emph{Forty years of muckenhoupt weights}, Function Spaces
  and Inequalities (Paseky nad Jizerou, June 2013), Lecture Notes (J. Lukeš
  and L. Pick, eds.), Matfyzpress, Prague, pp. 23–75.

\bibitem{duomartinombrosi2016}
J.~Duoandikoetxea, F.~J. Mart\'{\i}n-Reyes, and S.~Ombrosi, \emph{On the
  {$A_\infty$} conditions for general bases}, Math. Z. \textbf{282} (2016),
  no.~3-4, 955--972.

\bibitem{franciosi1989}
M.~Franciosi, \emph{Weighted rearrangements and higher integrability results},
  Studia Math. \textbf{92} (1989), no.~2, 131--139.

\bibitem{franciosi1991}
\bysame, \emph{On weak reverse integral inequalities for mean oscillations},
  Proc. Amer. Math. Soc. \textbf{113} (1991), no.~1, 105--112.

\bibitem{gurov1975}
L.~G. Gurov, \emph{The stability of {L}orentz transformations. {E}stimates for
  the derivatives}, Dokl. Akad. Nauk SSSR \textbf{220} (1975), 273--276.

\bibitem{gurovreshetnyak1976}
L.~G. Gurov and Yu.~G. Reshetnyak, \emph{A certain analogue of the concept of a
  function with bounded mean oscillation}, Sibirsk. Mat. Z. \textbf{17} (1976),
  no.~3, 540--546.

\bibitem{heinonen}
J.~Heinonen, \emph{Lectures on analysis on metric spaces}, Universitext,
  Springer-Verlag, New York, 2001.

\bibitem{hytonenperezrela2012}
T.~Hyt\"{o}nen, C.~P\'{e}rez, and E.~Rela, \emph{Sharp reverse {H}\"{o}lder
  property for {$A_\infty$} weights on spaces of homogeneous type}, J. Funct.
  Anal. \textbf{263} (2012), no.~12, 3883--3899.

\bibitem{IndratnoMaldonadoSilwal2015}
S.~Indratno, D.~Maldonado, and S.~Silwal, \emph{A visual formalism for weights
  satisfying reverse inequalities}, Expo. Math. \textbf{33} (2015), no.~1,
  1--29.

\bibitem{iwaniec1982}
T.~Iwaniec, \emph{On {$L^{p}$}-integrability in {PDE}s and quasiregular
  mappings for large exponents}, Ann. Acad. Sci. Fenn. Ser. A I Math.
  \textbf{7} (1982), no.~2, 301--322.

\bibitem{KinnunenKurkiMudarra2022}
J.~Kinnunen, E.-K. Kurki, and C.~Mudarra, \emph{Characterizations of weak
  reverse {H}\"{o}lder inequalities on metric measure spaces}, Math. Z.
  \textbf{301} (2022), no.~3, 2269--2290.

\bibitem{kinnunenshukla2014a}
J.~Kinnunen and P.~Shukla, \emph{Gehring's lemma and reverse {H}\"{o}lder
  classes on metric measure spaces}, Comput. Methods Funct. Theory \textbf{14}
  (2014), no.~2-3, 295--314.

\bibitem{kinnunenshukla2014b}
\bysame, \emph{The structure of reverse {H}\"{o}lder classes on metric measure
  spaces}, Nonlinear Anal. \textbf{95} (2014), 666--675.

\bibitem{korenovskii1990}
A.~A. Korenovskyy, \emph{The connection between mean oscillations and exact
  exponents of summability of functions}, Mat. Sb. \textbf{181} (1990), no.~12,
  1721--1727.

\bibitem{Korenovskii2007}
\bysame, \emph{Mean oscillations and equimeasurable rearrangements of
  functions}, Lecture Notes of the Unione Matematica Italiana, vol.~4,
  Springer, Berlin; UMI, Bologna, 2007.

\bibitem{KorenovskyyLernerStokolos2002}
A.~A. Korenovskyy, A.~K. Lerner, and A.~M. Stokolos, \emph{A note on the
  {G}urov-{R}eshetnyak condition}, Math. Res. Lett. \textbf{9} (2002), no.~5-6,
  579--583.

\bibitem{KorenovskyyLernerStokolos2007}
\bysame, \emph{A note on the maximal {G}urov-{R}eshetnyak condition}, Ann.
  Acad. Sci. Fenn. Math. \textbf{32} (2007), no.~2, 461--470.

\bibitem{kortekansanen2011}
R.~Korte and O.~E. Kansanen, \emph{Strong {$A_\infty$}-weights are
  {$A_\infty$}-weights on metric spaces}, Rev. Mat. Iberoam. \textbf{27}
  (2011), no.~1, 335--354.

\bibitem{myyrylainen}
K.~Myyryl\"{a}inen, \emph{Median-type {J}ohn--{N}irenberg space in metric
  measure spaces}, J. Geom. Anal. \textbf{32} (2022), no.~4, 131.

\bibitem{reshetnyak1994}
Yu.~G. Reshetnyak, \emph{Stability theorems in geometry and analysis},
  Mathematics and its Applications, vol. 304, Kluwer Academic Publishers Group,
  Dordrecht, 1994.

\bibitem{sawyer1982}
E.~T. Sawyer, \emph{Two weight norm inequalities for certain maximal and
  integral operators}, Harmonic analysis ({M}inneapolis, {M}inn., 1981),
  Lecture Notes in Math., vol. 908, Springer, Berlin-New York, 1982,
  pp.~102--127.

\bibitem{Spadaro2012}
E.~Spadaro, \emph{Nondoubling ${A}_\infty$ weights}, Adv. Calc. Var. \textbf{5}
  (2012), no.~3, 345--354.

\bibitem{strombergtorchinsky1989}
J.-O. Str\"{o}mberg and A.~Torchinsky, \emph{Weighted {H}ardy spaces}, Lecture
  Notes in Mathematics, vol. 1381, Springer-Verlag, Berlin, 1989.

\bibitem{wik1987}
I.~Wik, \emph{Note on a theorem by {R}eshetnyak-{G}urov}, Studia Math.
  \textbf{86} (1987), no.~3, 287--290.

\end{thebibliography}
\end{document}